\documentclass{article}

\usepackage{comment}
\usepackage{amsmath}       
\usepackage{amssymb}       
\usepackage{amsthm}

\usepackage{chapterbib}    
\usepackage{color}
\usepackage{comment}
\usepackage{bm,bbm}
\usepackage{overpic}

\usepackage{times}

\usepackage{epsfig}
\usepackage{graphicx,graphics,rotating}
\usepackage{subfigure}
\usepackage{multicol}
\usepackage{multirow}

\definecolor{MyDarkGreen}{rgb}{0,0.45,0}

\def\trait #1 #2 #3 {\vrule width #1pt height #2pt depth #3pt}
\def\fin{\hfill
        \trait .3 5 0
        \trait 5 .3 0
        \kern-5pt
        \trait 5 5 -4.7
        \trait 0.3 5 0
\medskip}

\newcommand{\Pj}{\Pi_{\hh,\P}}

\newcommand{\REAL}{\mathbbm{R}}

\newcommand{\X}{\mathbf{x}}

\newcommand{\TEXTFONTD}   [1]{#1}

\newcommand{\HONE}      {\TEXTFONTD{H}^1}
\newcommand{\HONEzr}    {\TEXTFONTD{H}^1_0}
\newcommand{\HTWO}      {\TEXTFONTD{H}^2}

\newcommand{\LTWO}      {\TEXTFONTD{L}^2}

\newcommand{\HS}[1]     {\TEXTFONTD{H}^{#1}}
\newcommand{\CS}[1]     {\TEXTFONTD{C}^{#1}}
\newcommand{\PS}[1]     {\TEXTFONTD{\mathbbm{P}}_{#1}}

\newcommand{\TEXTFONTA}[1]{\mathsf{#1}}

\renewcommand{\P} {\TEXTFONTA{P}}           
\newcommand{\E} {\TEXTFONTA{e}}           
\newcommand  {\V} {\TEXTFONTA{v}}           

\newcommand{\hh}{h}
\newcommand{\Tmesh}{\Omega_{\hh}}

\newcommand{\TEXTFONTB}[1]{\mathcal{#1}}
\newcommand{\Pset}{\TEXTFONTB{P}}          
\newcommand{\Vset}{\TEXTFONTB{V}}          
\newcommand{\hP}{\hh_{\P}}

\newcommand{\mP}{\ABS{\P}}

\newcommand{\mE}{\ABS{\E}}

\newcommand{\NMB}{N}
\newcommand{\NP}{\NMB^{\Pset}}   

\newcommand{\NPV}{\NMB^{\Vset}_{\P}}      

\newcommand{\dV}   {\,dV}
\newcommand{\dS}   {\,dS}

\newcommand{\GRAD} {\nabla}

\newcommand{\nor}  {\mathbf{n}}

\newcommand{\norP} {\mathbf{n}_{\P}}

\newcommand{\qs}{q}

\newcommand{\us}{u}
\newcommand{\ush}{\us_{\hh}}

\newcommand{\vs}{v}
\newcommand{\vsh}{\vs_{\hh}}

\newcommand{\ws}{w}
\newcommand{\wsh}{\ws_{\hh}}

\newcommand{\Fs} {F}
\newcommand{\Fsh}{\Fs_{\hh}} 

\newcommand{\vv} {\mathbf{v}}

\newcommand{\fs} {f}
\newcommand{\gs} {g}

\newcommand{\scal}   [2]{\big(#1,#2\big)}

\newcommand{\bilA}   [2]{\mathcal{A}\big(#1,#2\big)}
\newcommand{\bilAP}  [2]{\mathcal{A}_{\P}\big(#1,#2\big)}

\newcommand{\bilSP}  [2]{\mathcal{S}_{\hh,\P}\big(#1,#2\big)}

\newcommand{\bilAh}  [2]{\mathcal{A}_{\hh}\big(#1,#2\big)}

\newcommand{\bilAhP} [2]{\mathcal{A}_{\hh,\P}\big(#1,#2\big)}

\newcommand{\abs}    [1]{\big|#1\big|}

\newcommand{\ABS}    [1]{\left|#1\right|}

\newcommand{\gD}{\gs}

\newcounter{cst}

\newcommand{\vh}{\vv_{\hh}}

\newcommand{\NDG}     {m}

\newcommand{\matB}{\mathsf{B}}

\newcommand{\matD}{\mathsf{D}}

\newcommand{\matG}{\mathsf{G}}
\newcommand{\matI}{\mathsf{I}}

\newcommand{\matM}{\mathsf{M}}

\newcommand{\matGt}{\mathsf{\widetilde{G}}}

\newcommand{\Pn}  {\Pi^{\nabla}}
\newcommand{\Pnk} {\Pn_{\NDG}}

\newcommand{\matPn} {{\bm\Pi^{\nabla,\phi}_{\NDG}}}
\newcommand{\matPnk}{{\bm\Pi^\nabla_{\NDG}}}

\newcommand{\xP}{x_{\P}}
\newcommand{\yP}{y_{\P}}

\newcommand{\Th}{\Omega_{\hh}}

\newcommand{\Vh} {V_{\hh}}
\newcommand{\Vhg}{V_{\hh,g}}
\newcommand{\Vhz}{V_{\hh,0}}
\newcommand{\VhP}{V_{\hh,\P}}

\newcommand{\restrict}[2]{{#1}{}_{|{#2}}}

\theoremstyle{plain}

\newtheorem{theorem}{Theorem}[section]

\theoremstyle{definition}

\theoremstyle{plain}

\newcommand{\EOD}{\end{document}}

\newcounter{numbs}
\newcounter{numbi}
\newcounter{numbii}

\newcommand{\NewItem}{
  \stepcounter{numbi}
  \setcounter{numbii}{0}
  \bigskip\noindent
  \textbf{\thenumbs.\thenumbi.}
}

\newcommand{\SubItem}{
  \stepcounter{numbii}
  \medskip\noindent
  \textbf{\thenumbs.\thenumbi.\thenumbii.}
}

\newcommand{\SecItem}[1]
{
  \setcounter{numbi} {0}
  \stepcounter{numbs}
  \bigskip\medskip\noindent
  \begin{large}
  \textbf{\!\!\thenumbs.~~#1}
  \end{large}
}

\newcommand{\PGRAPH}[1]{\noindent\textbf{#1}}

\author{G. Manzini}
\title{Annotations on the virtual element method\\ for second-order elliptic problems}
\date{}

\begin{document}

\maketitle

\SecItem{Introduction}

\NewItem
This document contains working annotations on the Virtual Element
Method (VEM) for the approximate solution of diffusion problems with
variable coefficients.
To read this document you are assumed to have familiarity with
concepts from the numerical discretization of Partial Differential
Equations (PDEs) and, in particular, the Finite Element Method (FEM).
This document is \emph{not} an introduction to the FEM, for which many
textbooks (also free on the internet) are available.
Eventually, this document is intended to evolve into a \emph{tutorial
  introduction} to the VEM (but this is really a long-term goal).

\SubItem
To ease the exposition, we will refer to the Laplace problem and
consider the conforming method for the primal form in two space
dimensions.

\NewItem
Acronyms used in this document:
\begin{itemize}
\item \textbf{FEM} = \textbf{F}inite  \textbf{E}lement \textbf{M}ethod
\vspace{-0.5\baselineskip}
\item \textbf{VEM} = \textbf{V}irtual \textbf{E}lement \textbf{M}ethod
\vspace{-0.5\baselineskip}
\item \textbf{MFD} = \textbf{M}imetic \textbf{F}inite \textbf{D}ifference method
\vspace{-0.5\baselineskip}
\item \textbf{PFEM} = \textbf{P}olygonal \textbf{F}inite \textbf{E}lement \textbf{M}ethod
\vspace{-0.5\baselineskip}
\item \textbf{PDE} = \textbf{P}artial \textbf{D}ifferential \textbf{E}quation
\end{itemize}

\NewItem
The Virtual Element Method is a kind of Finite Element Methods where
trial and test functions are the solutions of a PDE problem inside
each element.

\SubItem
A \emph{naive} approach to implement the VEM would consists in solving
each local PDE problem to compute the trial and test functions at
least approximately where needed (for example, at the quadrature nodes
of the element).

\NewItem
The key point in the VEM approach is that some elliptic and $\LTWO$
polynomial projections of functions and gradients of functions that
belong to the finite element space are computable exactly using only
their degrees of freedom.
So, the VEM strategy is to substitute test and trial functions in the
bilinear forms and linear functionals of the variational formulations
with their polynomial projections whenever these latters are
computable.
In fact, the value of such bilinear forms is exact whenever at least
one of the entries is a polynomial function of a given degree.
We refer to this exactness property as the \emph{polynomial
  consistency} or, simply, the \emph{consistency} of the method.
This strategy leads to an underdetermined formulation (the global
stiffness matrix is rank deficient).
In the VEM we fix this issue by adding a stabilization term to the
bilinear form.
Stabilization terms are required not to upset the exactness property
for polynomials and to scale properly with respect to the mesh size
and the problem coefficients.

\SubItem
The low-order setting for polygonal cells (in 2D) and polyhedral cells
(in 3D) can be obtained as a straightforward generalization of the
FEM.
The high-order setting deserves special care in the treatment of
variable coefficients in order not to loose the optimality of the
discretization and in the 3D formulation.

\NewItem
The shape functions are virtual in the sense that we never compute
them explicitly (not even approximately) inside the elements.
Note that the polynomial projection of the approximate solution is
readily available from the degrees of freedom and we can use it as 
the numerical solution. 

\NewItem
The major advantage offered by the VEM is in the great flexibility in
admitting unstructured polygonal and polyhedral meshes.

\NewItem
The VEM can be seen as an evolution of the MFD method.
Background material (and a few historical notes) will be given in the
final section of this document.

\SecItem{Meshes}

\NewItem
\PGRAPH{Meshes with polygonal \& polyhedral cells.}

The VEM inherits the great flexibility of the MFD method and is
suitable to approximate PDE problems on unstructured meshes of cells
with very general geometric shapes.
The mesh flexibility is a major feature of this kind of
discretization, as we often need that computational meshes should be
easily adaptable to:
\begin{itemize}
\item the geometric characteristics of the domain;
\vspace{-0.5\baselineskip}
\item the solution.
\end{itemize}

\SubItem
Examples of 2D meshes used in academic problems are given in
Figure~\ref{fig:figure:mesh:2D}.
\begin{figure}[h]
  \begin{center}
    \begin{tabular}{c}
      \includegraphics[scale=0.125]{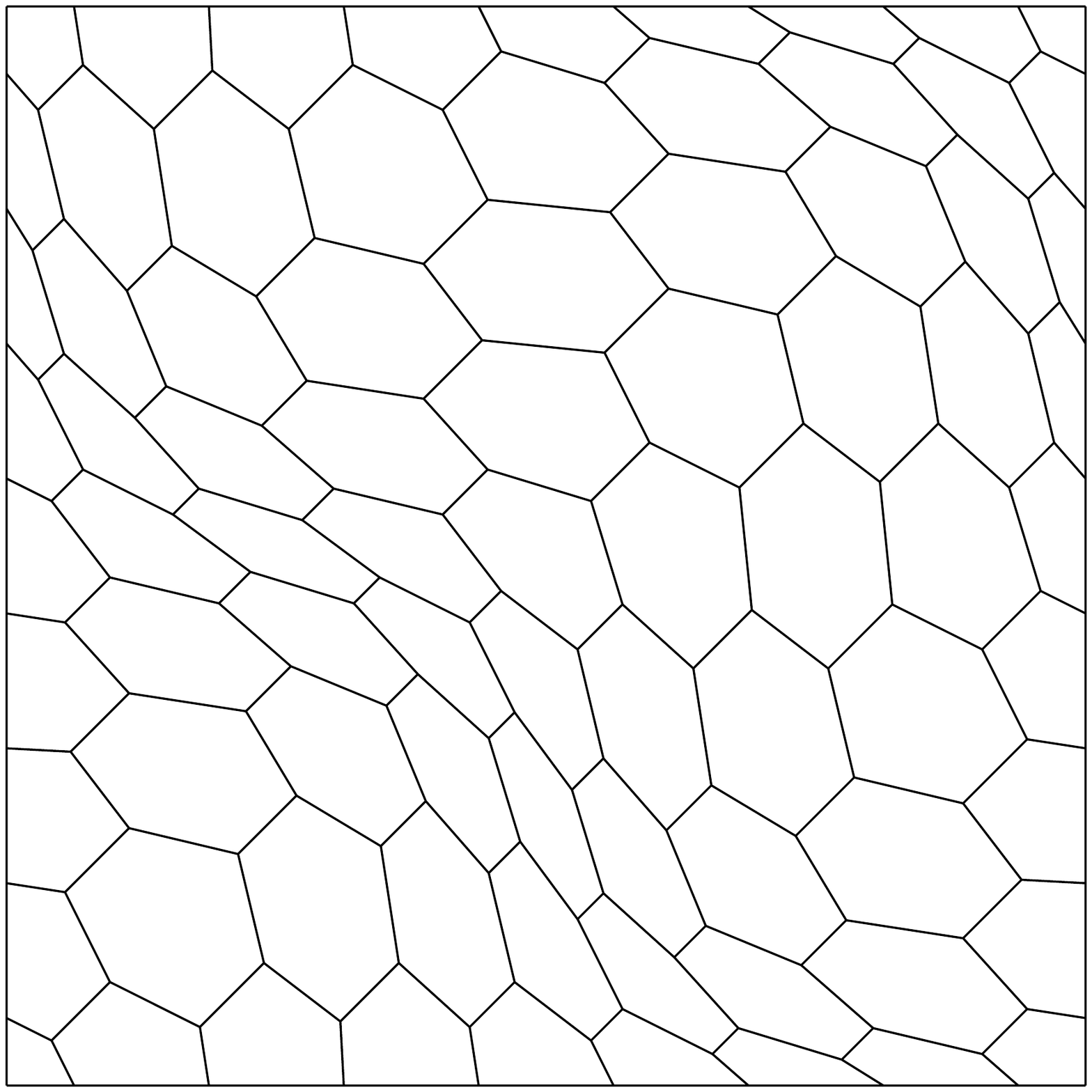}
      \includegraphics[scale=0.125]{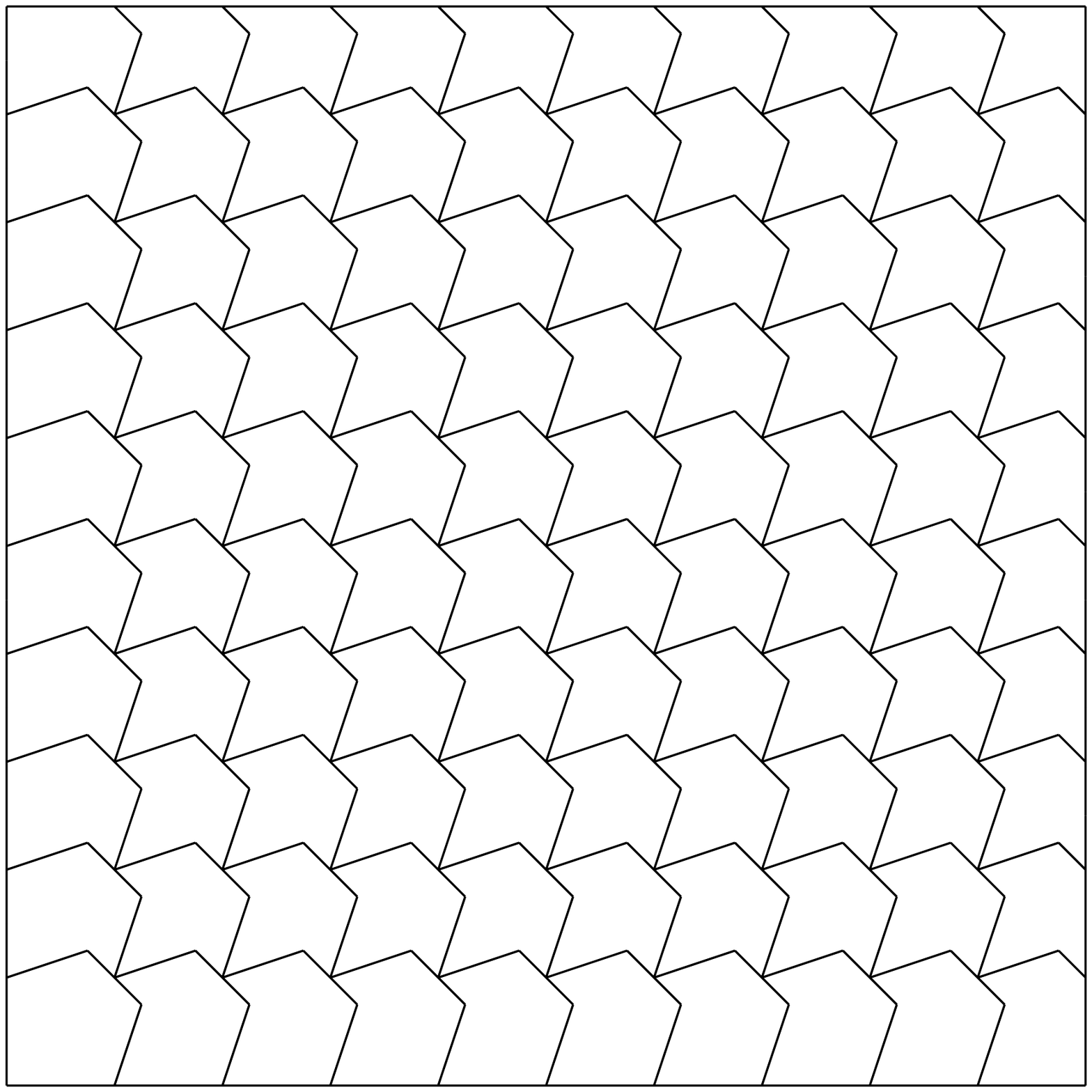}
      \includegraphics[scale=0.125]{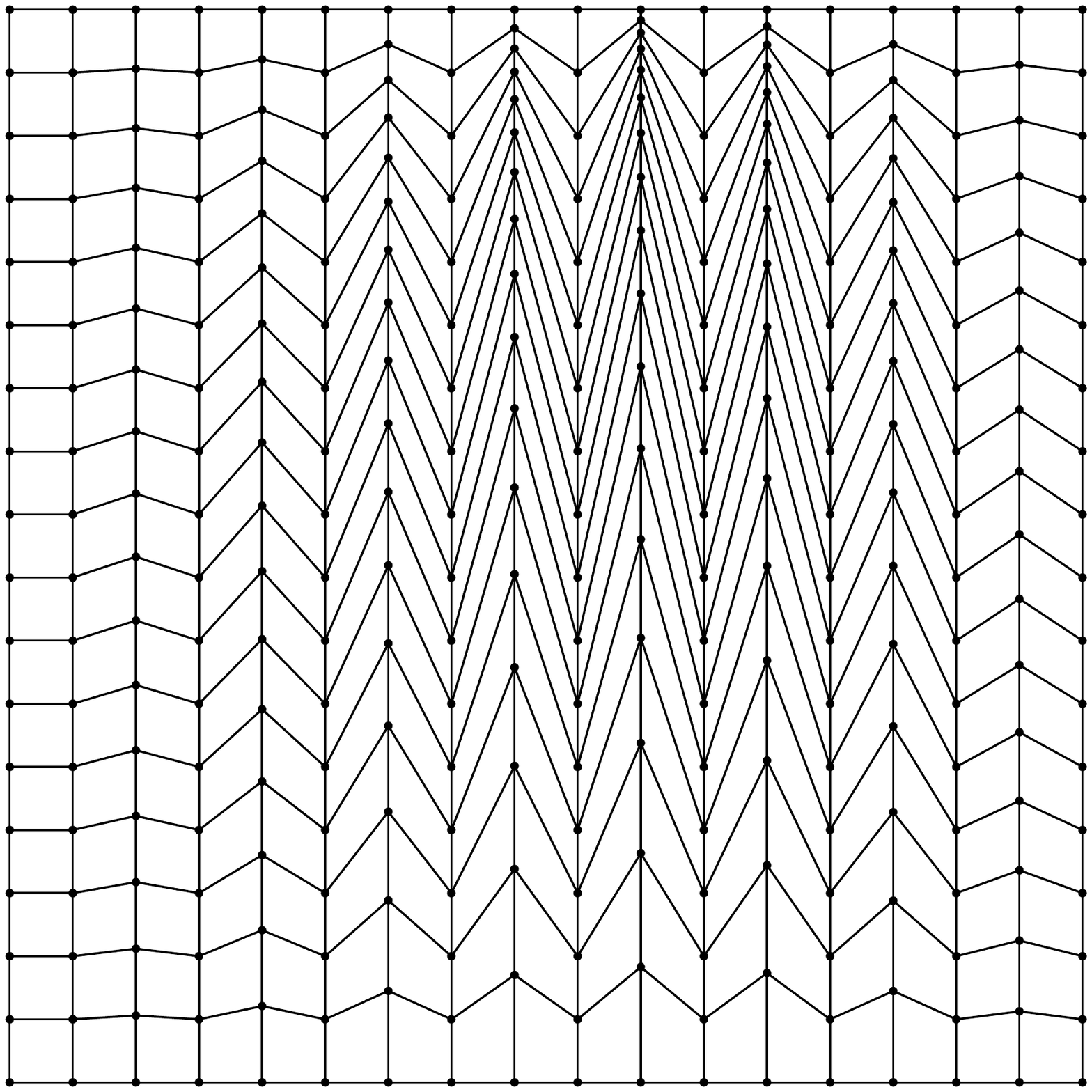}
      \includegraphics[scale=0.125]{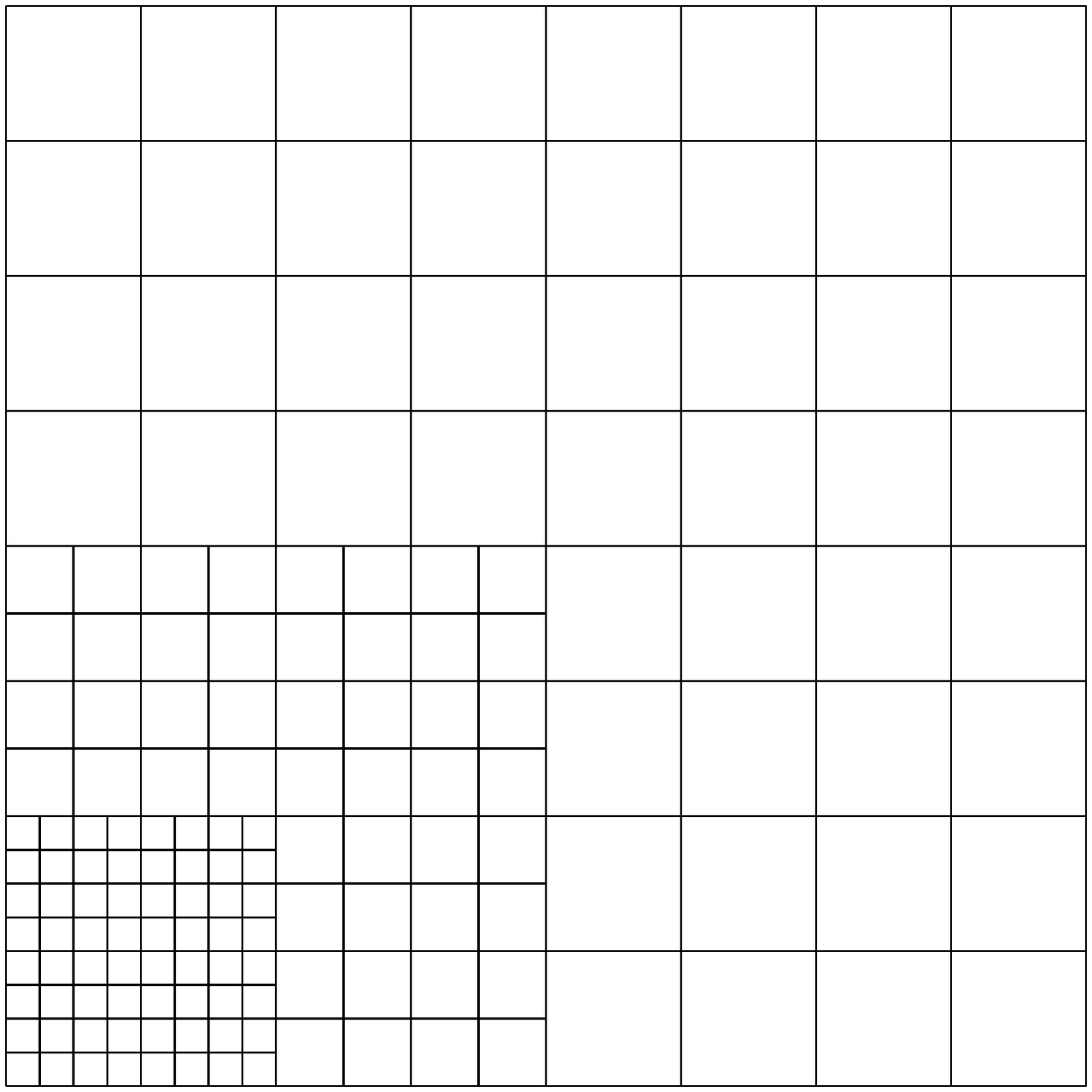}
    \end{tabular}
  \end{center}
  \vspace{-1.2cm}
  \caption{Convex and non-convex cells; highly deformed cells; local
    adaptive refinements (AMR, hanging nodes).}
  \label{fig:figure:mesh:2D}
\end{figure}

\SubItem
Examples of 3D meshes used in academic problems are given in
Figure~\ref{fig:figure:mesh:3D}.
\begin{figure}[h]
  \begin{center}
    \begin{tabular}{c}
      \hspace{-0.5cm}
      \includegraphics[scale=0.15]{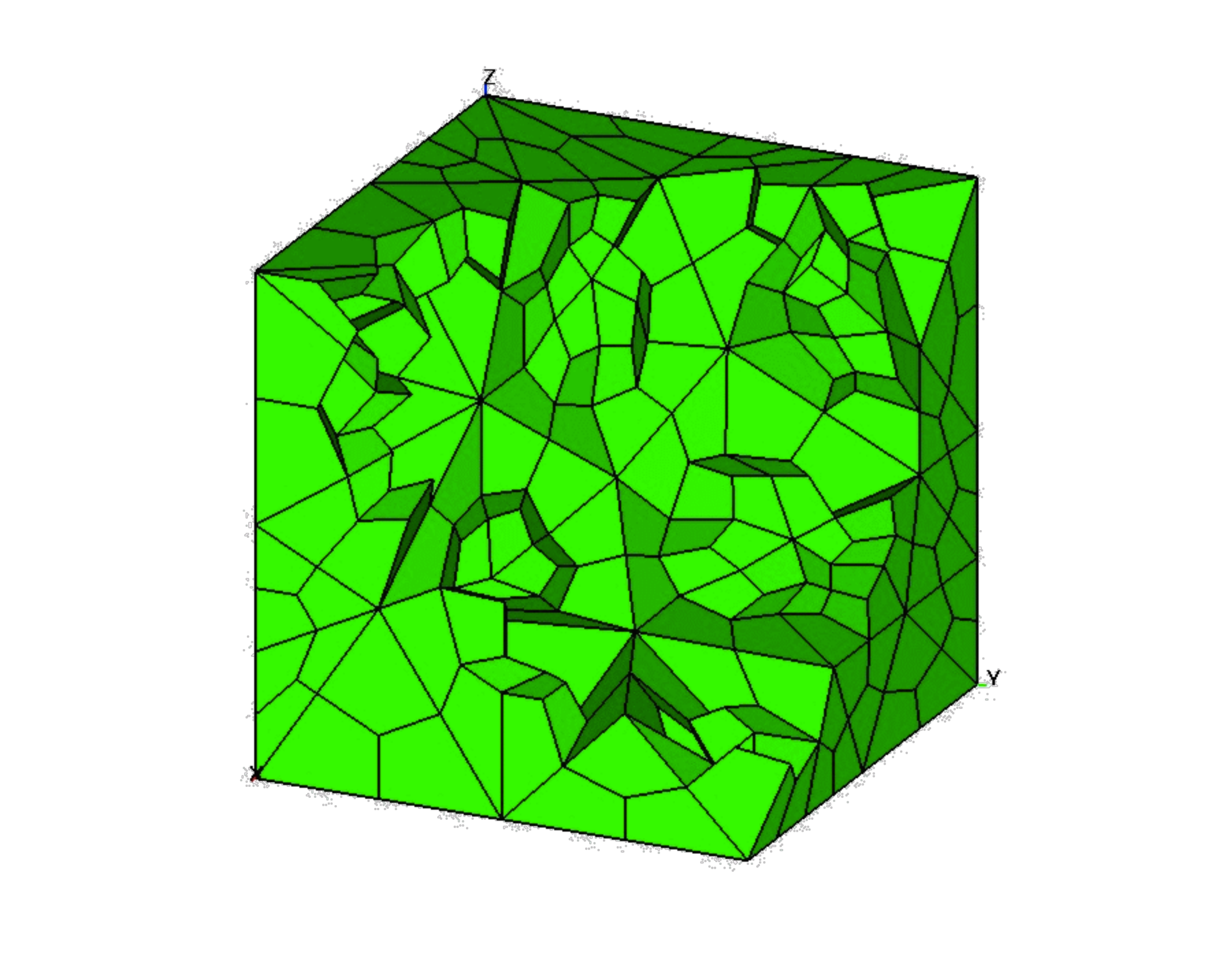}\hspace{-1cm}
      \includegraphics[scale=0.15]{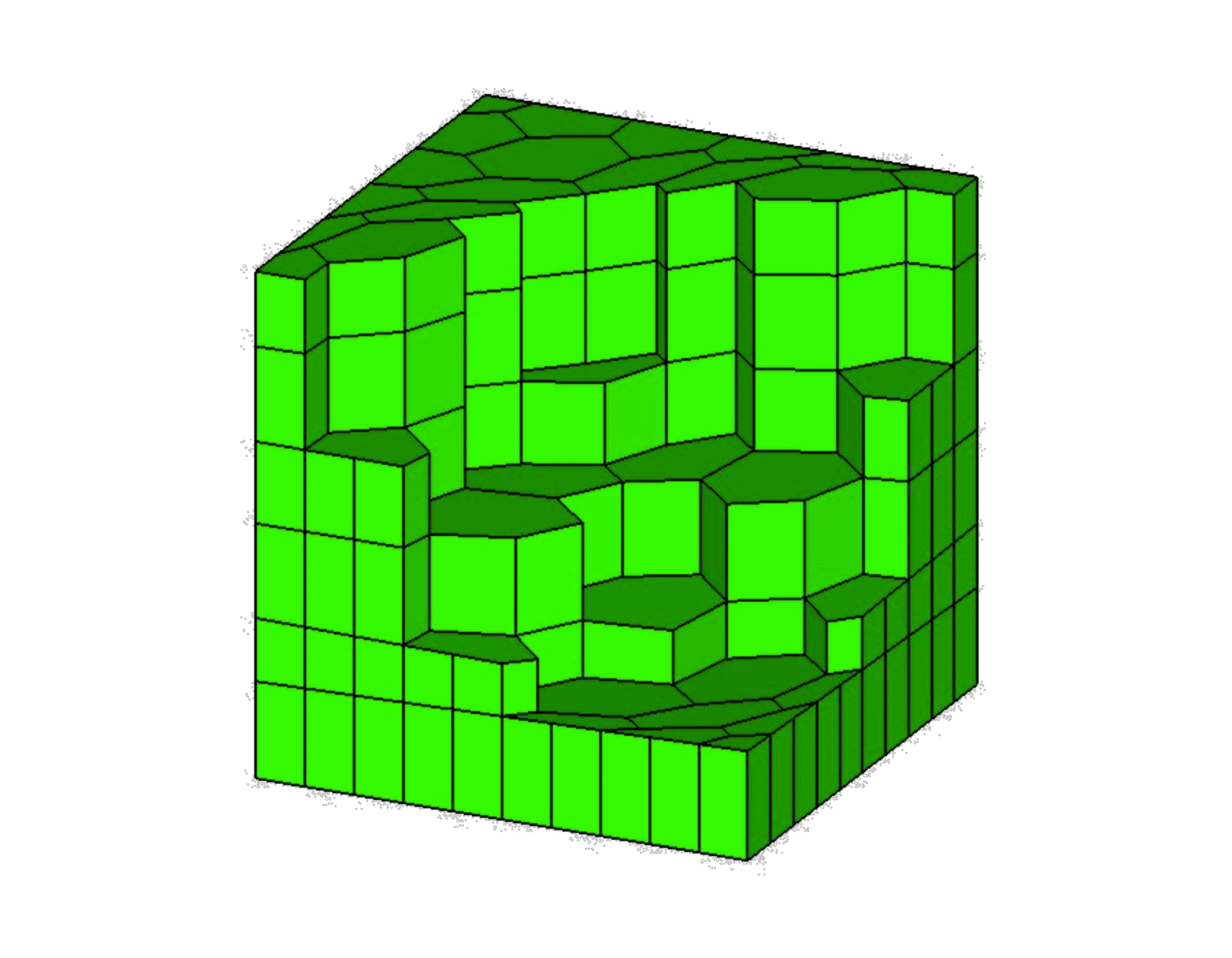}\hspace{-1cm}
      \includegraphics[scale=0.15]{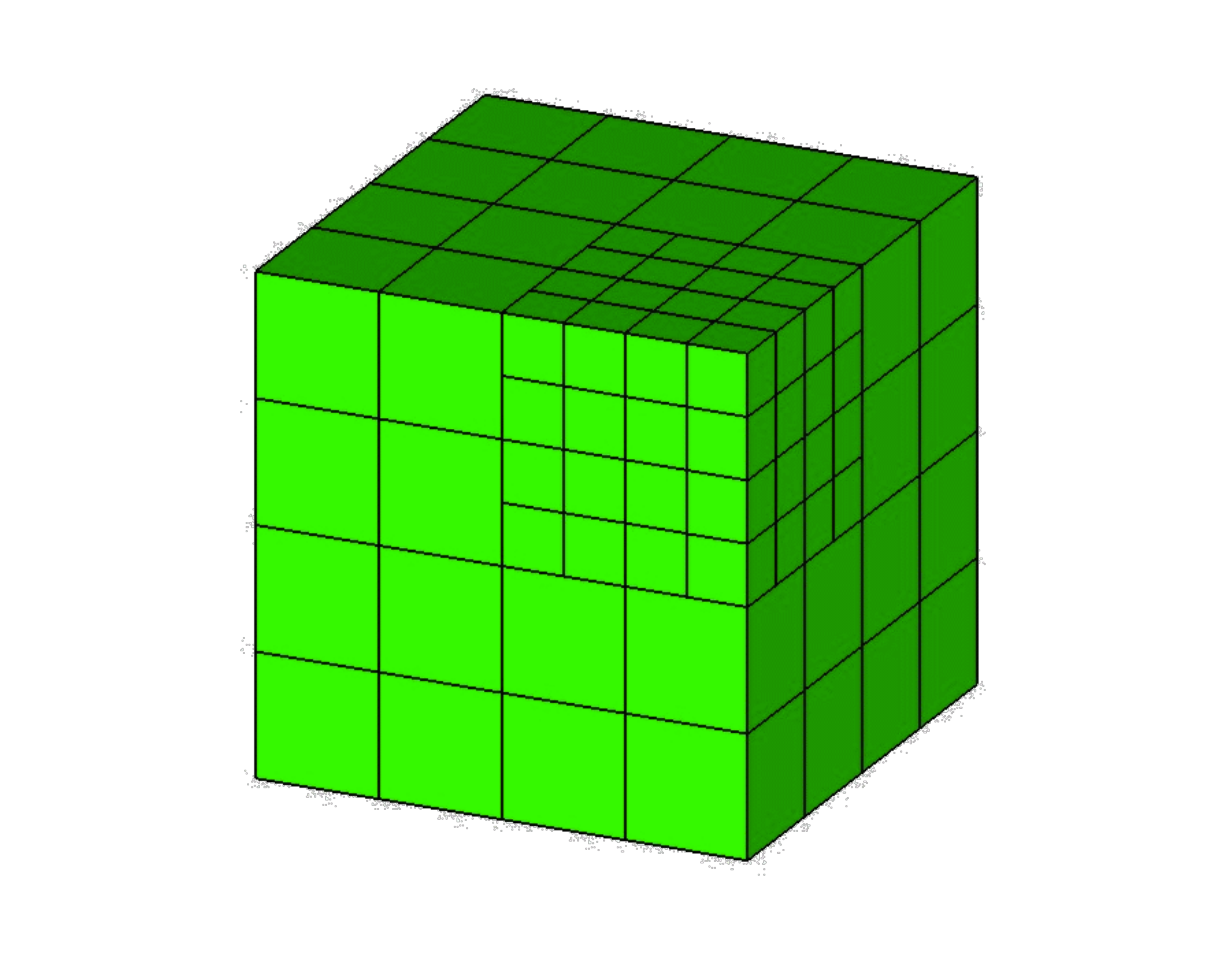}
    \end{tabular}
  \end{center}
  \vspace{-1cm}
  \caption{Random hexahedra; prysmatic meshes; local adaptive
    refinements (AMR, hanging nodes).}
  \label{fig:figure:mesh:3D}
\end{figure}

\SecItem{Formulations}

To introduce the basic ideas of the virtual element method, we first
consider the case of the Laplace equation.
Let $\Omega\subset\REAL^d$ ($d=2$) be a polygonal domain with boundary
$\Gamma$.
We use standard notation on Sobolev (and Hilbert) spaces, norms, and
semi-norms.
Hereafter, $\HONE_{\gs}(\Omega)$ denotes the affine Sobolev space of
functions in $\LTWO(\Omega)$ with first derivatives in
$\LTWO(\Omega)$, and boundary trace equal to $\gs$;
$\HONEzr(\Omega)$ is the subspace of $\HONE(\Omega)$
corresponding to $\gs=0$.
Also, we assume that $\fs\in\LTWO(\Omega)$ and
$\gD\in\HS{\frac{1}{2}}(\Gamma)$.

\NewItem
\PGRAPH{Differential formulation.}
\emph{Find $\us\in\HONE(\Omega)$ such that:}
\begin{align*}
  -\Delta\us     &= \fs\phantom{\gD~}\textrm{in~~}\Omega,\\[0.125em]
  \us            &= \gD\phantom{\fs~}\textrm{on~~}\Gamma.
\end{align*}

\NewItem
\PGRAPH{Variational formulation.}
\emph{Find $\us\in\HONE_{\gs}(\Omega)$ such that:}
\begin{align*}
  \bilA{\us}{\vs}:=\int_{\Omega}\GRAD\us\cdot\GRAD\vs\dV=\int_{\Omega}\fs\vs\dV=:\scal{\fs}{\vs}
  \qquad\forall\vs\in\HONEzr(\Omega).
\end{align*}

\NewItem
\PGRAPH{Finite element formulation.}
\emph{Find $\ush\in\Vhg\subset\HONE_{\gs}(\Omega)$ such that:}
\begin{align*}
  \bilA{\ush}{\vsh}:=\int_{\Omega}\GRAD\ush\cdot\GRAD\vsh\dV=\int_{\Omega}\fs\vsh\dV=:\scal{\fs}{\vsh}
  \quad\forall\vsh\in\Vhz\subset\HONEzr(\Omega),
\end{align*}
where $\Vhg$ and $\Vhz$ are finite dimensional affine and linear
functional spaces, dubbed \emph{the finite element spaces} (see
Section~4 for the formal definition).

\NewItem
\PGRAPH{A short guide to a happy life (with VEM)}

Let $\P$ be a cell of a mesh partitioning $\Th$ of domain $\Omega$.
The construction of the VEM follows these four logical steps.

\begin{itemize}
\item[$(i)$] Define the functions of the virtual element space
  $\Vh(\P)$ as the solution of a PDE problem;

\item[$(ii)$] Split the virtual element space as the direct sum of a
  the polynomial subspace $\PS{\NDG}(\P)$ of the polynomials of degree
  up to $\NDG$ and its complement:\\
  \begin{center}
    \begin{tabular}{rll}
      $\Vh(\P)$ = \big[\text{polynomials}\big] $\oplus$ \big[\text{non-polynomials}\big].
    \end{tabular}
  \end{center}
  At this point, it is natural to introduce the polynomial projection
  operator $\Pi_{\NDG}^{\nabla}:\Vh(\P)\to\PS{\NDG}(\P)$ and
  reformulate the splitting above as:
  \begin{align*}
    \Vh(\P) = \Pi^{*}_{\NDG}\big( \Vh(\P) \big) \oplus \big(1-\Pi^{*}_{\NDG}\big)\big( \Vh(\P) \big).
  \end{align*}

\item[$(iii)$] Define the virtual element approximation of the local
  bilinear form as the sum of two terms, respectively related to the
  consistency and the stability of the method:
  \begin{align*}
    \int_{\P}\nabla\ush\cdot\nabla\vsh\dV \approx 
    \text{\big[CONSISTENCY\big]} + \text{\big[STABILITY\big]}
    =: \bilAhP{\ush}{\vsh}.
  \end{align*}

\item[$(iv)$] Prove the convergence theorem:\\[1em]
  \centerline{
    \big[CONSISTENCY\big] + \big[STABILITY\big] $\Rightarrow$ \big[CONVERGENCE\big]
  }  
\end{itemize}

\SecItem{Extending the linear conforming FEM from triangles to
  polygonal cells}

\NewItem
\PGRAPH{The linear conforming Galerkin FEM on triangles.}
For every cell $\P$, we define the \textbf{local finite element
space} as:
\begin{align*}
  \VhP := \Big\{
  \vsh\in\HONE(\P) \,|\, \vsh\in\PS{1}(\P)
  \Big\}.
\end{align*}
Local spaces glue gracefully and yield the \textbf{global} finite
element space:
\begin{align*}
  \Vh := \big\{
  \vsh\in\HONE(\Omega) \,|\, \vsh{}_{|\P}\in\VhP\,\,\forall\P\in\Tmesh
  \big\}\subset\HONE(\Omega),
\end{align*}
which is a conforming approximation of $\HONE(\Omega)$.

\SubItem
We take into account the essential boundary conditions through the
affine space
\begin{align*} 
  \Vhg := \Big\{
  \vsh\in\VhP \,|\, \vsh{}_{|\Gamma}=\gs
  \Big\}
\end{align*}
for the given boundary function $\gs\in\HS{\frac{1}{2}}(\Gamma)$.
We take $\gs=0$ in the definition above and consider the linear
subspace $\Vhz\subset\Vh$.
The affine space $\Vhg$ and the linear space $\Vhz$ are respectively
used for the trial and the test functions in the variational
formulation.

\SubItem
The \textbf{degrees of freedom} of a virtual function $\vsh\in\Vh$ are
the values at mesh vertices.
Unisolvence can be proved and conformity of the global space is
obvious.

\NewItem
\PGRAPH{An extension of the FEM to polygonal cells.}
We extend the FEM to polygonal meshes by redefining the \emph{local
  finite element space} $\VhP$ on the polygonal cell $\P$.
To this end, we require that
    
\begin{itemize}  
\item the degrees of freedom of each function are its vertex values;
  therefore, the dimension of this functional space on cell $\P$ is
  equal to the number of vertices of the cell: $\text{dim}\VhP=\NPV$;
\item on triangles $\VhP$ must coincide with the linear Galerkin
  finite element space.
  This condition implies that $\VhP$ contains the linear polynomials
  $1$, $x$, $y$ and all their linear combinations;
\item the local spaces $\VhP$ \textit{glue gracefully} to give a
  \textbf{conformal} finite element space $\Vh$ of functions globally
  defined on $\Omega$.
\end{itemize}

\NewItem
\PGRAPH{Shape functions on polygons.}
The construction of a set of shape functions defined on $\P$ and
satisfying the requirements listed above is hard if we assume that
$\P$ may have a general polygonal geometric shape.
The challenge here is to define $\VhP$ by specifying the minimal
information about its functions so that a \emph{computable variational
  formulation} is possible.
To achieve this task, we specify the behavior of the functions of
$\VhP$ on $\partial\P$, for example by assuming that their trace on
each polygonal boundary is a piecewise polynomial of an assigned
degree.
Then, we characterize them inside $\P$ in a very indirect way by
assuming that they solve a partial differential equation.

\NewItem
\PGRAPH{The local finite element space on a generic polygonal cell.}
We define the local finite element space $\VhP$ as the span of a set
of shape functions $\varphi_i\in\HONE(P)$, which are associated with
the cell vertices $\V_i$.
A possible construction of these shape functions is given in the
following three steps.

\medskip
\begin{itemize}
\item[1.]  For each vertex $\V_i$ consider the boundary function
  $\delta_i$ such that:
  \begin{itemize}
  \item $\delta_i(\V_j)=1$ if $i=j$, and $0$ otherwise;
  \item $\delta_i$ is continuous on $\partial\P$ ;
  \item $\delta_i$ is linear on each edge.
  \end{itemize}
      
  \medskip
\item[2.] Then, set $\quad\restrict{\varphi_i}{\partial\P}=\delta_i$;
  
  \medskip
\item[3.] Finally, consider the function $\varphi_i$ that is 
  the \textbf{harmonic lifting} of the function $\delta_i$ inside $\P$.
\end{itemize}

\medskip
Eventually, we set: 
\begin{align*}
  \VhP:=\text{span}\{\varphi_1,\dots,\phi_{\NP}\}
\end{align*}

\NewItem
\PGRAPH{The harmonic lifting.}

We define the shape function $\varphi_i$ associated with vertex $i$ as
the harmonic function on $\P$ having $\delta_i$ as trace on boundary
$\partial\P$.
Formally, $\varphi_i\in\HONE(\P)$ is the solution of the following
elliptic problem:
\begin{equation*}
  \begin{aligned}
    -\Delta \varphi_i &= 0\phantom{\delta_i}  \quad\text{in~}\P, \\
    \varphi_i         &= \delta_i\phantom{0}  \quad\text{on~}\partial\P.
  \end{aligned}  
\end{equation*}

\SubItem
\PGRAPH{Properties of $\varphi_i$.}
Note that:
\begin{itemize}
\item the functions $\{\varphi_i\}$ are linearly independent (in the
  sense of the Gramian determinant);
\item if $\wsh\in\VhP$, then
  $\wsh=\sum_{i=1}^{\NP}\wsh(\V_i)\,\varphi_i$;
\item linear polynomials belong to $\VhP$, i.e., $1,x,y\in\VhP$;
\item the trace of each $\varphi_i$ on each edge of $\partial\P$ only
  depends on the vertex values of that edge.
  Consequently, the local spaces $\VhP$ glue together giving a
  conformal finite element space $\Vh\subset\HONEzr(\Omega)$.
\end{itemize}

\SubItem
Moreover, we can easily prove that:
\begin{itemize}
\item if $\P$ is a triangle, we recover the ${\mathbb P}_1$ Galerkin
  elements;
\item if $\P$ is a square, we recover the ${\mathbb Q}_1$ bilinear
  elements.
\end{itemize}

\NewItem
\PGRAPH{The low-order local virtual element space.} 
The \textbf{low-order} local virtual element space ($\NDG=1$) is given
by:
\begin{align*}
  \Vh^{1}(\P) = \Big\{ \vsh\in\HONE(\P)\,:\,
  & \Delta\vsh=0, \\
  & {\vsh}_{|\E}\in\PS{1}(\E)\quad\forall\E\in\partial\P,\\
  & {\vsh}_{|\partial\P}\in\CS{0}(\partial\P) \Big\}
\end{align*}

\SubItem
\PGRAPH{Extension of the conforming virtual space to the high-order case.}
The \textbf{high-order} local VE space ($\NDG\geq 2$) reads as:
\begin{align*}
  \Vh^{\NDG}(\P) = \Big\{ \vsh\in\HONE(\P)\,:\,
  & \Delta\vsh\in\PS{\NDG-2}(\P), \\
  & {\vsh}_{|\E}\in\PS{\NDG}(\E)\quad\forall\E\in\partial\P,\\
  & {\vsh}_{|\partial\P}\in\CS{0}(\partial\P) \Big\}
\end{align*}

\SubItem
The local spaces glue gracefully to provide a \emph{conforming}
approximation of $\HONE(\Omega)$.
A natural way to guarantee the last property in the previous
definitions, i.e., ${\vsh}_{|\partial\P}\in\CS{0}(\partial\P)$, is to
consider the \emph{vertex values} in the set of the degrees of
freedom.

\NewItem
\PGRAPH{The Harmonic Finite Element Method.}
Using the shape function of the previous construction, we can define a Polygonal 
Finite Element Method (PFEM), which we may call the \emph{Harmonic FEM}.
This method formally reads as: \emph{Find $\ush\in\Vh$ such that}
\begin{align*}
  \bilA{\ush}{\vsh} = \Fsh(\vsh) \quad\textrm{for~all~}\vsh\in\Vh,
\end{align*}
where (as usual)
\begin{align*}
  \bilA{\ush}{\vsh} = \int_{\Omega}\nabla\ush\cdot\nabla\vsh\dV,
\end{align*}
and $\Fsh(\vsh)$ is some given approximation of
$\displaystyle\int_{\Omega}\fs\vs\dV$, which we assume computable from
the given function $\fs$ by using only the vertex values of $\vsh$.
A possible approximation for the right-hand side is given in the
following step.

\SubItem
\PGRAPH{Low-order approximation of the right-hand side.}
The low-order approximation of the forcing term is given by:
\begin{align*}
  \scal{\fs}{\vsh} = 
  \sum_{\P}\int_{\P}\fs\vsh\dV \approx
  \sum_{\P}\overline{\vs}_{\hh.\P}\int_{\P}\fs\dV
  =: \Fsh(\vsh)
  \qquad\overline{\vs}_{\hh,\P}=\frac{1}{N_{\P}}\sum_{i=1}^{N_{\P}}\vsh(\X_i)
\end{align*}

\SubItem
In the literature, harmonic shape functions are used, for example, in
computer graphics (see, for example,~\cite{Martin:2008}) and
elasticity (see~\cite{Bishop:2014}).

\NewItem
\PGRAPH{Mesh assumptions.}
Under \emph{reasonable} assumptions on the mesh and the sequence of
meshes for $\hh\to 0$, the harmonic finite element discretization of
an elliptic problem enjoys the usual convergence properties.
We do not enter into the details here; among these properties we just
mention that
\begin{itemize}
\item all geometric objects must scale properly with the
  characteristic lenght of the mesh; for example, $\mP\simeq\hh^2$,
  $\mE\simeq\hh$, etc;
      
  \medskip
\item each polygon must be \emph{star-shaped} with respect to an
  internal ball of points $\Rightarrow$ \emph{interpolation estimates
    hold} (see~\cite{Brenner-Scott:2002,DiPietro-Ern:2011});
  more generally, each polygon can be the union of a \emph{uniformly
    bounded} number of star-shaped subcells.
\end{itemize}

\NewItem
\PGRAPH{Computability issues.}
The Harmonic FEM is a very nice method that works on polygonal meshes
and has a solid mathematical foundation.
This method is expected to be second order accurate when it
approximates sufficiently smooth solutions.
In such a case, the approximation error is expected to scale down
proportionally to $\hh$ in the energy norm and proportionally to
$\hh^2$ in the $\LTWO$ norm.
However, we need to compute the shape functions to compute the
integrals forming the stiffness matrix
\begin{align*}
  \bilA{\varphi_i}{\varphi_j} = \int_\Omega\nabla\varphi_i\cdot\nabla\varphi_j\dV
\end{align*}
and the term on the right-hand side
\begin{align*}
  \Fsh(\vh) = \int_{\Omega}\fs\varphi_i\dV.
\end{align*}
A numerical approximation of the shape functions inside $\P$ is
possible for example by partitioning the cell in simplexes and then
applying a standard FEM.
However, this approach is rather expensive since it requires the
solution of a Laplacian problem on element $\P$ for each one of its
vertices.

\SecItem{The virtual approximation of the bilinear form
  $\bilA{\varphi_i}{\varphi_j}$}

The virtual element approach follows a different strategy by resorting
to the so-called \emph{virtual construction} of the bilinear form.
Roughly speaking, we assume that the bilinear form $\mathcal{A}$ of
the Harmonic FEM is approximated by a virtual bilinear form
$\mathcal{A}_{\hh}$.
Since the bilinear form $\mathcal{A}$ can be split into elemental
contributions $\mathcal{A}_{\P}$
\begin{align*}
  \bilA{\vsh}{\wsh} =
  \sum_{P}\bilAP{\restrict{\vs}{P}}{\restrict{\ws}{P}}=\sum_{P}\int_P\nabla\vs\cdot\nabla\ws\dV,
\end{align*}
we assume that $\mathcal{A}_{\hh}$ is split into elemental
contributions $\mathcal{A}_{\hh,\P}$.
Each $\mathcal{A}_{\hh,\P}$ is a local approximation of the
corresponding $\mathcal{A}_{\P}$.
Therefore, we have that:
\begin{align*}
  \bilAh{\vsh}{\wsh}=\sum_{P}\bilAhP{\restrict{\vsh}{P}}{\restrict{\wsh}{P}},
  \qquad\textrm{with~}\mathcal{A}_{\hh,\P}\approx\mathcal{A}_{\P}.
\end{align*}

\NewItem
Now, we give \emph{two} conditions that must be satisfied by each
local bilinear form: \textbf{consistency} and \textbf{stability}.
These two conditions are inherited from the MFD formulation and as for
this latter they guarantee the convergence of the method.

\NewItem
\PGRAPH{Consistency.}
Consistency is an exactness property for linear polynomials.
Formally, for all $\qs\in{\mathbb P}_1(\P)$ and for all $\vsh\in\VhP$:
\begin{align*}
  \bilAhP{\vsh}{\qs}=\bilA{\vsh}{\qs}.
\end{align*}

\NewItem
\PGRAPH{Stability.} 
There exist two positive constants $\alpha^*$ and $\alpha_*$
independent of $\P$, such that
\begin{align*}
  \alpha_*\bilAP{\vsh}{\vsh}\leq\bilAhP{\vsh}{\vsh}\leq\alpha^*\bilAP{\vsh}{\vsh}.
\end{align*}

\NewItem
\PGRAPH{Convergence Theorem.}

\setcounter{section}{1}
\begin{theorem}
  Assume that for each polygonal cell $\P$ the bilinear form
  $\mathcal{A}_{\hh,\P}$ satisfies the properties of
  \emph{consistency} and \emph{stability} introduced above.
  Let $\ush\in\Vh$ be such that
  \begin{align*}
    \bilAh{\ush}{\vsh}=\Fsh(\vsh) 
    \qquad\emph{for~all~}\vsh\in\Vh.
  \end{align*}
  Then, 
  \begin{align*}
    \boxed{\phantom{\Big(}\|\us-\ush\|_{\HONE(\Omega)} \le C\hh
      \|\us\|_{\HTWO(\Omega)} \phantom{\Big)}}
  \end{align*}
\end{theorem}
\begin{proof}
  See~\cite{BeiraodaVeiga-Brezzi-Cangiani-Manzini-Marini-Russo:2013}.
\end{proof}

\NewItem
\PGRAPH{ A good starting point to build $\mathcal{A}_{\hh,\P}$\ldots leading to a rank-deficient stiffness matrix. } 


\SubItem
We know the functions $\vsh$ of $\VhP$ only on the boundary of $\P$
and we can compute the \emph{exact value} of the quantity
\begin{align*}
  \overline{\nabla\vsh} := \dfrac{1}{|\P|}\int_{\P}\nabla\vsh\dV
\end{align*}
using only the vertex values.
In fact,
\begin{align*}
  \int_{\P}\nabla\vsh\dV
  = \int_{\partial P}\vsh\norP\dS
  = \sum_{i=1}^{\NP} \left(\int_{\E_i}\vsh\dS\right)\nor_{\P,i}
  = \sum_{i=1}^{\NP}\dfrac{\vsh(\V_i)+\vsh(\V_{i+1})}{2}|\E_i|\nor_{\P,i}
\end{align*}
$\overline{\nabla\vsh}$ is a constant vector in $\REAL^2$. 

\SubItem
Now, we are really tempted to say that
\begin{align*}
  \bilAP{\varphi_i}{\varphi_j} :=
  \int_{\P}\nabla\varphi_i\cdot\nabla\varphi_j\dV
  \approx
  \int_{\P}\overline{\nabla\varphi_i}\cdot\overline{\nabla\varphi_j}\dV
  =:\bilAhP{\varphi_i}{\varphi_j}
\end{align*}
Note that if $\P$ is a triangle we obtain the stiffness matrix of the
linear Galerkin FEM.
However, $\bilAhP{\varphi_i}{\varphi_j}$ would have \emph{rank $2$}
for any kind of polygons, thus leading to a \emph{rank deficient}
approximation for $\mathcal{A}_{\hh}$ on any mesh that is not
(strictly) triangular!

\renewcommand{\Pj}{\Pi^{\nabla}_{\NDG}}

\NewItem
\PGRAPH{The local projection operator $\Pj$.}
We define the \emph{local projection operator} for each polygonal cell
$\P$
\begin{align*}
  \Pj\,:\,\VhP \longrightarrow \PS{1}(\P)
\end{align*}
that has the two following properties:
\begin{itemize}
\item[$(i)$] it approximates the gradients using only the vertex values:
\begin{align*}
  \nabla\left(\Pj\vsh\right) = \overline{\nabla\vsh};
\end{align*}
\vspace{-0.5\baselineskip}
\item[$(ii)$] it preserves the linear polynomials:
\begin{align*}
  \Pj\qs = \qs\quad\text{for~all~}\qs\in\PS{1}(\P).
\end{align*} 
\end{itemize}

\NewItem
\PGRAPH{The bilinear form $\mathcal{A}_{\hh,\P}$.}
We start writing that 
\begin{align*}
  \bilAhP{\ush}{\vsh}=\bilAhP{\Pj\ush}{\vsh}+\bilAhP{\ush-\Pj\ush}{\vsh}.
\end{align*}
With an easy computation it can be shown that
\begin{align*}
  \bilAhP{\Pj\ush}{\vsh} = \bilAP{\Pj\ush}{\Pj\vsh} 
  := \mathcal{A}_{\hh,\P}^{0}\big(\ush,\vsh\big)
\end{align*}
and
\begin{align*}
  \bilAhP{(I-\Pj)\ush}{\vsh} = \bilAP{(I-\Pj)\ush}{(I-\Pj)\vsh}
  \longrightarrow\mathcal{A}_{\hh,\P}^{1}\big(\ush,\vsh\big).
\end{align*}
We will set:     
\begin{align*}
  \boxed{ 
    \phantom{\Big(}
    \mathcal{A}_{\hh,\P} = \mathcal{A}_{\hh,\P}^{0}+\mathcal{A}_{\hh,\P}^{1} =
    \text{CONSISTENCY} + \text{STABILITY}
    \phantom{\Big)}
  }
\end{align*}

\NewItem
\PGRAPH{The consistency term $\mathcal{A}_{\hh,\P}^{0}$.}
The bilinear form $\mathcal{A}_{\hh,\P}^{0}$ provides a sort of
\emph{``constant gradient approximation''} of the stiffness matrix.
In fact, $\mathcal{A}_{\hh,\P}^{0}$ ensures the \emph{consistency
  condition: $\bilAhP{\vsh}{\qs}=\bilAP{\vsh}{\qs}$} for all
$\qs\in\PS{1}(\P)$; in fact,
\begin{align*}
  \mathcal{A}_{\hh,\P}^{0}\big(\vsh,\qs\big) & = 
  \int_{\P}\overline{\nabla\vsh}\cdot\overline{\nabla\qs}\dV =
  \mP\overline{\nabla\vsh}\cdot\overline{\nabla\qs} =
  \left(\int_{\P}\nabla\vsh\right)\cdot\overline{\nabla\qs}\dV
  \\[0.5em] &
  =
  \int_{\P}\nabla\vsh\cdot\overline{\nabla\qs}\dV =
  \int_{\P}\nabla\vsh\cdot\nabla\qs\dV =
  \bilAP{\vsh}{\qs}.
\end{align*}

\SubItem
The remaining term is zero because $(I-\Pj)\qs=0$ if
$\qs\in\PS{1}(\P)$.

\NewItem
\PGRAPH{The stability term $\mathcal{A}_{\hh,\P}^{1}$.}

We need to correct $\mathcal{A}_{\hh,\P}^{0}$ in such a way that:
\begin{itemize}
\item consistency is not upset;
  \vspace{-0.5\baselineskip}
\item the resulting bilinear form is stable;
  \vspace{-0.5\baselineskip}
\item the correction is computable using only the degrees of freedom!
\end{itemize}

In the founding paper it is shown that we can substitute the (non
computable) term $\bilAP{(I-\Pj)\ush}{(I-\Pj)\vsh}$ with
\begin{align*}
  \mathcal{A}_{\hh,\P}^{1}\big(\ush,\vsh\big) := \bilSP{ (I-\Pj)\ush }{ (I-\Pj)\vsh }
\end{align*}    
where $\mathcal{S}_{\hh,\P}$ can be \textbf{any symmetric and positive
  definite bilinear form} that behaves (asymptotically) like
$\mathcal{A}_{\P}$ on the kernel of $\Pj$.

Hence:
\begin{align*}
  \bilAhP{\ush}{\vsh} := 
  \underbrace{ \boxed{ \phantom{\Big(} \bilAP{\Pj\ush}{\Pj\vsh} \phantom{\Big)} } }_{\text{CONSISTENCY} }
  \quad+\quad
  \underbrace{ \boxed{ \phantom{\Big(} \bilSP{ (I-\Pj)\ush }{ (I-\Pj)\vsh } \phantom{\Big)} } }_{\text{STABILITY}}
\end{align*}

\SecItem{Implementation of the local virtual bilinear form in four
  steps}

\NewItem
The central idea of the VEM is that we can compute the polyomial
projections of the virtual functions and their gradients exactly using
only their degrees of freedom.
So, to implement the VEM we introduce the matrix representation of
such projection, i.e., the projection matrix.
The projection matrix representats the projection of the shape
functions with respect to a given monomial basis (or with respect to
the same basis of shape functions).

\NewItem
\PGRAPH{Step~1.}
To compute the projection matrix we need two special matrices called
$\matB$ and $\matD$.
Matrices $\matB$ and $\matD$ are constructed from the basis of the
polynomial subspace (the scaled monomials).
All other matrices are derived from straightforward calculations
involving $\matB$ and $\matD$.
The procedure is exactly the same for the conforming and nonconforming
method, except for the definition of matrix $\matB$.

\SubItem
\PGRAPH{Matrix $\matD$.}
To start the implementation of the VEM, we consider the matrices
$\matB$ and $\matD$.
The columns of matrix $\matD$ are the degrees of freedom of the scaled
monomials.
Since the scaled monomials are the same for both the conforming and
nonconforming case, the definition of matrix $\matD$ is also the same.
Matrix $\matD$ is given by:
\begin{align*}
  \matD =        
  \left[
    \begin{array}{ccc}
      1 & \dfrac{x_1-\xP}{\hP} & \dfrac{y_1-\yP}{\hP}\\[1em]
      1 & \dfrac{x_2-\xP}{\hP} & \dfrac{y_2-\yP}{\hP}\\[1em]
      {}& \ldots\\[1em]
      1 & \dfrac{x_n-\xP}{\hP} & \dfrac{y_n-\yP}{\hP}
    \end{array}
  \right].
\end{align*}

\SubItem
\PGRAPH{Matrix $\matB$.}
The column of matrix $\matB$ are the right-hand sides of the
projection problem.
Matrix $\matB$ is given by:
\begin{align*}
  \matB =
  \frac{1}{2\hP}
  \left[
    \begin{array}{cccc}
      \displaystyle\frac{1}{n}          & \displaystyle\frac{1}{n}          & \ldots &  \displaystyle\frac{1}{n} \\[1em]
      \abs{\E_1}n_{x,1}+\abs{\E_n}n_{x,n} & \abs{\E_2}n_{x,2}+\abs{\E_1}n_{x,1} & \ldots & \abs{\E_n}n_{x,n}+\abs{\E_{n-1}}n_{x,n-1} \\[1em]
      \abs{\E_1}n_{y,1}+\abs{\E_n}n_{y,n} & \abs{\E_2}n_{y,2}+\abs{\E_2}n_{y,1} & \ldots & \abs{\E_n}n_{y,n}+\abs{\E_{n-1}}n_{y,n-1}
    \end{array}
  \right].
\end{align*}

\NewItem
\PGRAPH{Step~2.}
Using these matrices, we compute matrix $\matG=\matB\matD$.

\NewItem
\PGRAPH{Step~3.}
Then, we solve the \textbf{projection problem} in algebraic form:
$\matG\matPnk=\matB$ for matrix $\matPnk$, which is the representation
of the projection operator $\Pnk$.

\NewItem
\PGRAPH{Step~4.}
Using matrices $\matPnk$ and $\matPn=\matD\matPnk$ we build the \textbf{stiffness
  matrix}:
\begin{align*}
  \matM = (\matPnk)^T\matGt\matPnk + \nu_{\P}(\matI-\matPn)^T\,(\matI-\matPn)
\end{align*}
where $\matGt$ is matrix $\matG$ with the first row set to zero,
$\nu_{\P}$ is a scaling coefficient.

\SecItem{Background material on numerical methods for PDEs using polygonal and polyhedral meshes}

\NewItem
The \textbf{conforming VEM} for the Poisson equation in primal form as
presented in the \emph{founding
  paper}~\cite{BeiraodaVeiga-Brezzi-Cangiani-Manzini-Marini-Russo:2013}
is a \emph{variational reformulation} of the nodal Mimetic Finite
Difference method of References~\cite{Brezzi-Buffa-Lipnikov:2009}
(low-order case) and~\cite{BeiraodaVeiga-Lipnikov-Manzini:2011}
(arbitrary order case).

\SubItem
The \textbf{nonconforming VEM} for the Poisson equation in primal form
as presented in the \emph{founding
  paper}~\cite{AyusodeDios-Lipnikov-Manzini:2016} is the variational
reformulation of the arbitrary-order accurate nodal MFD method of
Reference~\cite{Lipnikov-Manzini:2014}.

\SubItem
Incremental extensions are for different type of applications
(including parabolic problems) and formulations (mixed form of the
Poisson problem), hp refinement, and also different variants
(arbitrary continuity, discontinuous variants).
A complete and detailed review of the literature on MFD and VEM is out
of the scope of this collection of notes.
However, we list below a few references that can be of interest to a
reader willing to know more.

\SubItem
\textbf{Compatible discretization methods} have a long story. 
A review can be found in the conference paper~\cite{Bochev-Hyman:2006}.
A recent overview is also found in the first chapter of the Springer
book~\cite{BeiraodaVeiga-Lipnikov-Manzini:2014} on the MFD method for
elliptic problem and the recent paper~\cite{Palha-Gerritsma:2016}.

\SubItem
The state of the art on these topics is well represented in two recent
\textbf{special issues} on numerical methods for PDEs on unstructured
meshes~\cite{Bellomo-Brezzi-Manzini:2014,BeiraodaVeiga-Ern:2016}.

\SubItem
Comparisons of different methods are found in the \textbf{conference
  benchmarks} on 2D and 3D diffusion problms with anisotropic
coefficients,
see~\cite{Eymard-Henri-Herbin-Hubert-Klofkorn-Manzini:2011:FVCA6:benchmark:proc-peer}.

\SubItem
For the \textbf{Mimetic Finite Difference method}, we refer the interested
reader to the following works:

\begin{itemize}

\item extension to polyhedral cells with curved faces~\cite{Brezzi-Lipnikov-Shashkov-Simoncini:2007};~%

\item general presentation of the MFD method:%
  book~\cite{BeiraodaVeiga-Lipnikov-Manzini:2014}; %
  review papers~\cite{Gyrya-Lipnikov-Manzini-Svyatskiy:2014,Lipnikov-Manzini-Shashkov:2014};~%
  connection of MFD with other methods~\cite{Lipnikov-Manzini:2016:Durham:chbook};~%
  benchmarks~\cite{Lipnikov-Manzini:2011:FVCA6:benchmark:proc-peer,Manzini:2008:FVCA5:benchmark:proc-peer};~%
  conference paper~\cite{Manzini:2008:FVCA5:review:proc-peer},
  
\item advection-diffusion problems:
  connection with finite volume schemes~\cite{Beirao-Droniou-Manzini:2011};~%
  convergence analysis~\cite{Cangiani-Manzini-Russo:2009};~%

\item diffusion equation in primal form: 
  the low-order formulation~\cite{Brezzi-Buffa-Lipnikov:2009},%
  post-processing of solution and flux~\cite{BeiraodaVeiga-Manzini-Putti:2015};~%
  arbitrary-order of accuracy on polygonal~\cite{BeiraodaVeiga-Lipnikov-Manzini:2011,BeiraodaVeiga-Lipnikov-Manzini:2011:FVCA6:regular-paper:proc-peer} %
  and polyhedral meshes~\cite{Lipnikov-Manzini:2014};~%
  arbitrary regularity~\cite{BeiraodaVeiga-Manzini:2012:WCCM12:proc,BeiraodaVeiga:ECCOMAS12:proc};~%
  a posteriori estimators~\cite{Antonietti-BeiraodaVeiga-ovadina-Verani:2013};
  
\item diffusion equation in mixed form: 
  founding paper on the low-order formulation:
  convergence analysis~\cite{Brezzi-Lipnikov-Shashkov:2006} and 
  implementation~\cite{Brezzi-Lipnikov-Simoncini:2005};~%
  with staggered coefficients~\cite{Manzini-Lipnikov-Moulton-Shashkov:2016:SINUM:submitted,Lipnikov-Manzini-Moulton-Shashkov:2016};~%
  arbitrary order of accuracy~\cite{Gyrya-Lipnikov-Manzini:2016};~%
  monotonicity conditions and discrete maximum/minimum principles~\cite{Lipnikov-Manzini-Svyatskiy:2011,Lipnikov-Manzini-Svyatskiy:2011:FVCA6:regular-paper:proc-peer};~%
  eigenvalue calculation~\cite{Cangiani-Gardini-Manzini:2011};~%
  second-order flux discretization~\cite{BeiraodaVeiga-Manzini:2008a};~%
  convergence analysis~\cite{BeiraodaVeiga-Lipnikov-Manzini:2009};~%
  post-processing of flux~\cite{Cangiani-Manzini:2008};~%
  a posteriori analysis and grid adaptivity~\cite{BeiraodaVeiga-Manzini:2008b};~%
  flows in fractured porous media~\cite{Antonietti-Formaggia-Scotti-Verani-Verzott:2013b};

\item steady Stokes equations:
  the MFD method on polygonal meshes~\cite{BeiraodaVeiga-Gyrya-Lipnikov-Manzini:2009};~%
  error analysis~\cite{BeiraodaVeiga-Lipnikov-Manzini:2010};~%
  
\item differential forms and Maxwell eigenvalues~\cite{Brezzi-Buffa-Manzini:2014};~%
  magnetostatics~\cite{Lipnikov-Manzini-Brezzi-Buffa:2011};~%

\item quasilinear elliptic problems~\cite{Antonietti-Bigoni-Verani:2015}, 
  nonlinear and control problems~\cite{Antonietti-BeiraodaVeiga-Bigoni-Verani:2013,Antonietti-Bigoni-Verani:2013a},
  obstacle elliptic problem~\cite{Antonietti-BeiraodaVeiga-Verani:2013b},

\item topology optimization~\cite{Gain:2013}.~%

\end{itemize}

\SubItem
For the \textbf{Virtual Element Method}, we refer the interested
reader to the following works:

\begin{itemize}
\item first (founding) paper on the conforming VEM~\cite{BeiraodaVeiga-Brezzi-Cangiani-Manzini-Marini-Russo:2013};~%

\item diffusion problems in primal form:
  ``hp'' refinements ~\cite{BeiraodaVeiga-Chernov-Mascotto-Russo:2016}

\item diffusion problems in mixed form~\cite{Brezzi-Falk-Marini:2014,BeiraodaVeiga-Brezzi-Marini-Russo:2016c}

\item geomechanics~\cite{Andersen-Nilsen-Raynaud:2016}

\item $L^2$-projections~\cite{Ahmad-Alsaedi-Brezzi-Marini-Russo:2013}

\item steady Stokes equations: %
  stream function formulation~\cite{Antonietti-BeiraodaVeiga-Mora-Verani:2014};~%
  divergence free methods~\cite{BeiraodaVeiga-Lovadina-Vacca:2016};~%
  mixed VEM for the pseudostressvelocity formulation~\cite{Caceres-Gatica:2016}
  
\item fracture network
  simulations~\cite{Fumagalli-Keilegavlen:2016,Benedetto-Berrone-Scialo:2016,Benedetto-Berrone-Pieraccini-Scialo:2014,Benedetto-Berrone-Borio-Pieraccini-Scialo:2016a}

\item Cahn-Hilliard equation on polygonal meshes~\cite{Antonietti-BeiraodaVeiga-Scacchi-Verani:2016};~%
  
\item linear~\cite{BeiraodaVeiga-Brezzi-Marini:2013} and 
  nonlinear elasticity~\cite{BeiraodaVeiga-Lovadina-Mora:2015};~%
  plate bending problem~\cite{Brezzi-Marini:2013};~%
  finite deformations~\cite{Chi-BeiraodaVeiga-Paulino};~%
  three-dimensional elasticity~\cite{Gain-Talischi-Paulino:2014};

\item non-conforming formulation:%
  first (founding) paper~\cite{AyusodeDios-Lipnikov-Manzini:2016};~%
  Stokes~\cite{Cangiani-Gyrya-Manzini:2016};~%
  advection-reaction-equation~\cite{Cangiani-Manzini-Sutton:2016};~%
  biharmonic problems~\cite{Zhao-Chen-Zhang:2016,Antonietti-Manzini-Verani:2016:M3AS:submitted};~%

\item eigenvalue calculation~\cite{Gardini-Vacca:2016};~%
  Steklov eigenvalue problem~\cite{Mora-Rivera-Rodriguez:2016};~%

\item connection with other methods:
  PFEM~\cite{Manzini-Russo-Sukumar:2014,Natarajan-Bordas-Ooi:2015};
  MFD and BEM-based FEM~\cite{Cangiani-Gyrya-Manzini-Sutton:2017:GBC:chbook};

\item Helmholtz equation~\cite{Perugia-Pietra-Russo:2016};
  
\item implementation~\cite{BeiraodaVeiga-Brezzi-Marini-Russo:2014,Sutton:2016};

\item further developments of the conforming virtual
  formulation~\cite{BeiraodaVeiga-Brezzi-Marini-Russo:2016a,BeiraodaVeiga-Brezzi-Marini-Russo:2016d};~%
  stabilization of the hourglass phenomenon~\cite{Cangiani-Manzini-Russo-Sukumar:2015};~%
  VEM with with arbitrary regularity~\cite{BeiraodaVeiga-Manzini:2014,BeiraodaVeiga-Manzini:2012:WCCM12:proc};~%

\item advection-diffusion problems in the advection dominated
  regime~\cite{Benedetto-Berrone-Borio-Pieraccini-Scialo:2016b};

\item a posteriori analysis~\cite{BeiraodaVeiga-Manzini:2015};

\item topology optimization~\cite{Gain-Paulino-Leonardo-Menezes:2015};

\item contact problems~\cite{Wriggers-Rust-Reddy:2016}.

\end{itemize}

\NewItem
\PGRAPH{Polygonal/Polyhedral FEM.}
The development of numerical methods with such kind of flexibility or
independence of the mesh has been one of the major topics of research
in the field of numerical PDEs in the last two decades and a number of
schemes are currently available from the literature.
These schemes are often based on approaches that are substantially
different from MFDs or VEMs.

Recently developed discretization frameworks related to general meshes
include
\begin{itemize}
\item the finite element method using rational basis
  functions~\cite{Wachspress:1975,Wachspress:2010,Floater-Gillette-Sukumar:2014} and using
  generalized polynomial interpolants on polygons
  ~\cite{Sukumar-Tabarraei:2004,Sukumar-Malsch:2006,Tabarraei-Sukumar:2007}; 
\item the finite volume methods (see the review
  paper~\cite{Droniou:2014}) and the connection with the MFD method
  (see~\cite{Droniou-Eymard-Gallouet-Herbin:2010}); 
\item hybrid high-order (HHO) method~\cite{DiPietro-Ern:2014a,DiPietro-Ern:2014b,DiPietro-Ern:2015}; 
\item the discontinuous Galerkin (DG) method~\cite{DiPietro-Ern:2011}; 
\item hybridized discontinuous Galerkin (HDG) method~\cite{Cockburn-Gopalakrishnan-Lazarov:2009}; 
\item the weak Galerkin (wG) method \cite{Wang-Ye:2014}. 
\end{itemize}

\SubItem
Other examples (extended FEM, partition of unity, meshless, non-local
decomposition, etc) can be found in:
\cite{Arroyo-Ortiz:2007,
Babuska-Banerjee-Osborn:2003,
Babuska-Banerjee-Osborn:2004,
Babuska-Melenk:1997,
Babuska-Osborn:1983,
Bishop:2014,
Bochev-Hyman:2006,
DiPietro-Ern:2012,
Duarte-Babuska-Oden:2000,
Floater-Hormann-Kos:2006,
Fries-Belytschko:2010,
Gerstenberger-Wall:2008,
Idelsohn-Onate-Calvo-DelPin:2003,
Mohammadi:2008,
Rabczuk-Bordas-Zi:2010,
Rand-Gillette-Bajaj:2013,
Sukumar-Mos-Moran-Belytschko:2000,
Talischi-Paulino-Pereira-Menezes:2010,
Warren:1996,
Delzanno:2015,Manzini-Delzanno-Vencels-Markidis:2016,Camporeale-Delzanno-Lapenta-Daughton:2006,Vencels-Delzanno-Johnson-BoPeng-Laure-Markidis:2015,Camporeale-Delzanno-Bergen-Moulton:2015
}

\SubItem
Most of these methods use trial and test functions of a rather
complicate nature, that often could be computed (and integrated) only
in some approximate way.
In more recent times several other methods have been introduced in
which the trial and test functions are pairs of polynomial (instead of
a single non-polynomial function) or the degrees are defined on
multiple overlapping meshes:
Examples are found in:
\cite{Bonelle-Ern:2014,
Cockburn:2010,
Cockburn-Gopalakrishnan-Lazarov:2009,
Cockburn-Guzman-Wang:2009,
Droniou-Eymard-Gallouet-Herbin:2013,
Mu-Wang-Wang-Ye:2013,
Mu-Wang-Wei-Ye-Zhao:2013,
Wang-Ye:2013,
Wang-Ye:2013,
Coudiere-Manzini:2010,
Domelevo-Omnes:2005,
Droniou-Eymard:2006,
Hermeline:2000,
Hermeline:2003,
Hermeline:2007,%
Krell-Manzini:2012,%
Cockburn-DiPietro-Ern:2016%
}

\SecItem{Next versions.}
The next incremental developments of this document will cover:
\begin{itemize}
\item the high-order VEM formulation;
\vspace{-0.5\baselineskip}
\item the nonconforming formulation;
\vspace{-0.5\baselineskip}
\item connection with the MFD method;
\vspace{-0.5\baselineskip}
\item 3D extensions (enhancement of the virtual element space, etc).
\end{itemize}

\SecItem{Acknowledgements}
This manuscript is based on material that was disclosured 
in the last years under these status numbers: LA-UR-12-22744,
LA-UR-12-24336, LA-UR-12-25979, LA-UR-13-21197, LA-UR-14-24798,
LA-UR-14-27620.
This manuscript is assigned the disclosure number LA-UR-16-29660.
The development of the VEM at Los Alamos National Laboratory has been
partially supported by the Laboratory Directed Research and
Development program (LDRD), U.S. Department of Energy Office of
Science, Office of Fusion Energy Sciences, under the auspices of the
National Nuclear Security Administration of the U.S. Department of
Energy by Los Alamos National Laboratory, operated by Los Alamos
National Security LLC under contract DE-AC52-06NA25396.

\clearpage
\bibliographystyle{plain}
\bibliography{VEM,MFD,poly,pub-MFD,pub-VEM,delzanno}

\end{document}